\documentclass[12pt]{article}
\usepackage{a4wide}
\usepackage{euscript}
\usepackage{fullpage}
\usepackage[T2A]{fontenc}
\usepackage[utf8]{inputenc}
\usepackage[main=english, russian]{babel}
\usepackage{amsfonts}
\usepackage{amssymb, amsthm}
\usepackage{amsmath}
\usepackage{mathtools}
\usepackage{graphicx}
\usepackage{geometry}
\usepackage{tikz}
\usepackage{bbm}

\newtheorem{st}{Statement}
\newcommand{\Z}{\mathbb{Z}}
\newcommand{\R}{\mathbb{R}}

\newcommand{\bx}{\mathbf{x}}

\newcommand{\bz}{\mathbf{z}}

\newcommand{\D}{\mathbb D}
\newcommand{\Var}{\mathrm{Var}\,}
\newcommand{\conv}{\mathrm{conv}}

\newcommand{\area}{\mathrm{area}\,}

\newcommand{\vol}{\mathrm{vol}\,}
\newcommand{\cov}{\mathrm{cov}\,}

\newcommand{\deq}{\stackrel{D}{=}}
\newcommand{\ceq}{\stackrel{C}{=}}
\newcommand{\cdeq}{\stackrel{CD}{=}}
\newcommand{\E}{\mathbb E}
\newcommand{\ind}{\mathbbm 1}

\newtheorem{lm}{Lemma}
\newtheorem{tm}{Theorem}
\newtheorem{cor}{Corollary}
\newtheorem*{rem}{Remark}

\begin{document}
\title{On integer points inside a randomly shifted polyhedron}
\date{}
\author{Aleksandr Tokmachev\thanks{The work was supported by the Theoretical Physics and Mathematics Advancement Foundation «BASIS».}}
\maketitle
\begin{abstract}
Consider a convex body $C \subset \mathbb{R}^d$. Let $X$ be a random point with uniform distribution in $[0,1]^d$. Define $X_C$ as the number of lattice points in $\mathbb{Z}^d$ inside the translated body $C + X$. It is well known that $\mathbb{E} X_C = \mathrm{vol}(C)$.
A natural question arises: What can be said about the distribution of $X_C$ in general? In this work, we study this question when $C$ is a polyhedron with vertices at integer points.
\end{abstract}
\section{Introduction}
Consider the $d$-dimensional space $\mathbb R^d$ and the lattice $\mathbb Z^d$. Let $C$ be a convex body in $\mathbb R^d$, i.e., a convex compact set with non-empty interior. Let $X$ be a random point uniformly distributed inside the unit cube $[0,1]^d$. We denote by $X_C$ the number of lattice points inside $C + X$:
\begin{align}
    X_C = |(C + X) \cap \Z^d|, \nonumber
\end{align}
where $X$ is a random point uniformly distributed in $[0,1]^d$. It is known that the expectation of $X_C$ equals the volume of $C$:
\begin{align}
    \mathbb E X_C = \vol (C).
\end{align}

A natural question arises: what can be said about the variance of $X_C$? In the case when $C$ is a disk of large radius, this question is closely related to the \textit{circle problem}, which concerns finding the number of lattice points in a disk $rB^2$ of radius $r$ centered at the origin. For more details on this connection and possible generalizations, see \cite{Strombergsson}. Additionally, we should mention the paper \cite{Janacek}, which studies the asymptotics of the variance for $rC$ when the body $C$ differs from a ball.

In one of the author's works \cite{Myself}, the two-dimensional case of this problem was considered for integer polygons, i.e., polygons whose vertices all have integer coordinates. This paper provided an explicit formula for the variance, as well as computed the covariance of any two polygons. It is worth noting that both formulas have a simple form and express the characteristics in terms of the number of points on the sides. 

In this paper, we present generalizations of several two-dimensional results to higher dimensions, as well as provide counterexamples to other statements.

\section{Main results}
We will begin by establishing the notations that will be employed throughout this paper.

\textbf{Notations.}
\begin{itemize}
\item An \textit{integer point} is defined as a point whose coordinates are integers. Similarly, an \textit{integer vector or polyhedron} is defined as a vector or polyhedron with integer vertices. We identify integer vectors and integer points as follows: $m \leftrightarrow \overrightarrow{0m}$.

\item We use the standard notation $X \deq Y$ for identically distributed random variables and use $X \ceq Y$ for random variables such that $X - \E X = Y - \E Y $ a.s. Note that $X \ceq Y$ iff exist a constant $C$ such that $X = Y + C$ a.s. We also write $X \cdeq Y$ if $X - \E X \deq Y - \E Y$. 

\item For two sets $A$ and $B$ we denote by $A \oplus B$ the Minkowski sum of these sets:
$$  
    A \oplus B = \{a + b \colon a \in A,~b \in B\}.
$$


\item In the context of a polygon, the term ''sides'' is used to refer to a set of vectors representing the counterclockwise--oriented sides of the polygon. We call the \textit{affine length} of a side of an integer polygon the number of segments into which integer points divide that side.
\end{itemize}

Let us recall the key properties of quantities generated by two-dimensional polygons (for proofs, see \cite{Myself}).
\begin{st}\label{baza} Let an integer polygon $P$ with sides $v_1, v_2, \ldots v_n$ be given. Then:
    \begin{enumerate}
    
        \item $\E X_P = \area (P)$.

        \item For any $A \in SL(2, \Z)$, $$X_{AP} \deq X_P.$$

        \item $X_{-P} \deq X_P$.

        \item $X_P \ceq  X_{v_1} + X_{v_2} + \ldots + X_{v_n}.$

        \item For any integer polygons $P$ and $Q$,
        $$
            X_{P \oplus Q} \ceq X_P + X_Q, 
        $$
        where $P \oplus Q$ is a Minkowski sum of $P$ and $Q$.

         \item If $P$ is centrally symmetric, then $X_P = const$ a.s.

        \item $X_P \cdeq -X_p$, i.e. the random variable $X_P - \E X_P$ has a symmetric distribution. 
        
    \end{enumerate}
\end{st}

The first three properties remain valid for polyhedra of arbitrary dimension.
\begin{st}\label{baza_dim}
    Let P be an integer polyhedron. Then
    \begin{enumerate}
    
        \item $\E X_P = \vol (P)$.

        \item For any $A \in SL(d, \Z)$, $$X_{AP} \deq X_P.$$

        \item $X_{-P} \deq X_P$.
        
    \end{enumerate}
\end{st}

Property 4 was auxiliary in the two-dimensional case and has no analogue in higher dimensions. Moreover, Properties 5--7 fail already in dimension 3. Counterexamples for these properties are described in Section \ref{sec_Counterexamples}. Nevertheless, Property 6 has an analogue in arbitrary dimension. Before stating it, let us recall that a \textit{zonotope} is a polyhedron which is the Minkowski sum of several line segments.
\begin{st}\label{st_Zonotope}
    Let $P$ be an integer zonotope. Then $X _P = const$ a.s.
\end{st}
Note that in the two-dimensional case, the class of zonotopes coincides with the class of centrally symmetric polygons. Therefore, this result includes Property 6 from the two-dimensional case. The proof of Statement \ref{st_Zonotope} can be found in Section \ref{sec_Zonotope}.

\vspace{0,5 cm}
An important consequence of the properties of quantities generated by two-dimensional integer polygons is the behavior of the quantity's distribution under polyhedron scaling. Specifically, the following corollary holds:
\begin{cor}\label{cor_Scalling}
    Let an integer polygon $P$ be given. Then $X_{nP} \ceq nX_P$ and $\D X_{nP} = n^2\D X_P$.
\end{cor}
This corollary relies on Property 5 of the two-dimensional case, which fails in higher dimensions. This raises the question about the distributional behavior of the random variable generated by a polyhedron under scaling. Our following theorems provide the answer to this question. We first state the result for integer simplices.
\begin{tm}\label{tm_Scaling_simplex}
    Let $T$ be an integer simplex in $\R^d$. Then, there exist integer polyhedra $P_1, P_2, \ldots, P_d$, such that $P_1 = T$, $P_k = -P_{d + 1 - k}$ and for any  positive integer $n$
    \begin{align}
        X_{nT} = \sum_{k = 1}^d \binom{n - k + d}{d} X_{P_k} \quad a.s.
    \end{align}       
Moreover, $X_{P_1} + X_{P_2} + \ldots + X_{P_d} = const$.
\end{tm}

An arbitrary convex polyhedron can be represented as a union of simplices, allowing us to extend the previous theorem to the case of integer polyhedra. To achieve this, we first observe that the definition of the random variable $X_C$ naturally extends to cases where $C$ is a union of several disjoint bodies. In this setting, it is clear that if $C = \cup_{i = 1}^n C_i$, then $X_C = \sum_{i = 1}^n X_{C_i}$.
\begin{tm}\label{tm_Scaling_polytope}
    Let $P$ be an integer polyhedron in $\R^d$. Then, there exist unions of integer polyhedra $P_1, P_2, \ldots, P_d$, such that $P_1 = P$, $P_k = -P_{d + 1 - k}$ and for any positive integer $n$
    \begin{align}
        X_{nP} = \sum_{k = 1}^d \binom{n - k + d}{d} X_{P_k} \quad a.s.
    \end{align}
Moreover, $X_{P_1} + X_{P_2} + \ldots + X_{P_d} = const$.
\end{tm}

Let us note several corollaries of this theorem for the case of polyhedra in dimensions 3 and 4. Since the sum $X_{P_1} + X_{P_2} + \ldots + X_{P_d} = \text{const}$, we may subtract it from the right-hand side multiple times, causing the equality sign ``$=$'' to change to a sign ``$\ceq$''.
\begin{cor}\label{cor_3dim}
    Let $P \subset \R^3$ be an integer polyhedron. Then, for any positive integer $n$
    \begin{align}
        X_{nP} \ceq \frac{n(n+1)}{2}X_P - \frac{n(n-1)}{2}X_{-P}.
    \end{align}
\end{cor}
In particular, when $P$ is a centrally symmetric polygon, the behavior of the random variable follows the same pattern as in the planar case.
\begin{cor}\label{cor_3dim_sym}
    Let $P \subset \R^3$ be a centrally symmetric polyhedron. Then, $X_{nP} \ceq nX_P$ and $\D X_{nP} = n^2 \D X_P$ for any positive integer $n$.
\end{cor}
Similar arguments yield the result for 4-dimensional polyhedra.
\begin{cor}\label{cor_4dim_sym}
    Let $P \subset \R^4$ be a centrally symmetric polyhedron. Then, $X_{nP} \ceq n^2X_P$ and $\D X_{nP} = n^4 \D X_P$ for any positive integer $n$.
\end{cor}

Complete proofs of the theorems and their corollaries can be found in Section \ref{sec_Scaling}.

Let us return to the original question about the variance of the quantity generated by an integer polyhedron. In dimension 2, the variance of a polygon without parallel sides was equal to the sum of squared affine lengths of its sides divided by 12. In other words, the variance depended only on the combinatorial structure of the polygon: knowing the number of points on the sides was sufficient to compute the variance of the corresponding random variable. This raises the question of whether the variance of a polyhedron is determined by its combinatorial structure. The answer is negative, and a counterexample is provided by the so-called Reeve tetrahedron \cite{Reeve}: a tetrahedron with vertices at $(0,0,0)$, $(0,1,0)$, $(1,0,0)$, and $(1,1,n)$. Note that all tetrahedra of this form contain no integer points in their interior, on their faces, or edges. However, it can be easily shown that for large $n$, a translated Reeve tetrahedron may contain arbitrarily many integer points, and therefore we cannot expect the variances to be equal. The explicit form of variances for such tetrahedra is described in the following statement.
\begin{tm}
    Let \( T_n \) be the simplex with vertices \((0,0,0)\), \((0,1,0)\), \((1,0,0)\), and \((1,1,n)\). Then its variance is given by:
    \begin{align}
        \Var X_{T_n} =  \frac{n^3 + 12n - 3}{72n}.
    \end{align}
\end{tm}
The proof of this formula appears in Section~\ref{sec_Reeve}.

\section{Scaling}\label{sec_Scaling}
In this section we prove results related to scaling: Theorems~\ref{tm_Scaling_simplex} and~\ref{tm_Scaling_polytope}, along with their corollaries.

We say that polyhedra $P_1, P_2, \ldots, P_n$ form a partition of a polyhedron $P$ if $P = \bigcup_{i=1}^n P_i$ and $\text{int}(P_i \cap P_j) = \varnothing$ for any $i \ne j$. Note that if $P_1, P_2, \ldots, P_n$ is a partition of $P$, then $X_P = X_{P_1} + X_{P_2} + \ldots + X_{P_n}$ almost surely. Indeed, if some integer point lies simultaneously inside both $P_i$ and $P_j$, then it belongs to their intersection, and the probability of such an event is zero. We call a polyhedron $P'$ a copy of $P$ if $P'$ is obtained from $P$ by parallel translation. The key idea for proving the theorems is the following lemma.
\begin{lm}\label{lm_Scaling_symplex}
    For any simplex $T \subset \mathbb{R}^d$ (not necessarily integer), there exists a collection of $d$ convex polyhedra $P_1, P_2, \ldots, P_d$ such that:
    \begin{enumerate}
        \item $P_1 = T$ and $P_k = -P_{d + 1 - k}$ for all $k$;
        \item For any positive integer $n$, the scaled simplex $nT$ can be partitioned into copies of the polyhedra $P_1, P_2, \ldots, P_d$;
        \item Each polyhedron $P_k$ appears in the partition exactly $\binom{n - k + d}{d}$ times.
    \end{enumerate}
\end{lm}
\begin{proof}
Since any simplex can be obtained from any other simplex by an affine transformation, it suffices to prove the statement for the simplex $\Delta$ formed by the basis vectors:
\begin{align}
    \Delta = \{\bx = (x_1, x_2, \ldots, x_d) \in \mathbb{R}^d \colon x_1 + x_2 + \ldots + x_d \leqslant 1, x_i \geqslant 0\}.
\end{align}

We explicitly construct the required collection of polyhedra $P_1, P_2, \ldots, P_d$. To do this, we partition the unit cube $[0,1]^d$ by hyperplanes of the form $x_1 + x_2 + \ldots + x_d = k$ for $k = 1, 2, \ldots, d-1$. The resulting polyhedra have the following form:
\begin{align}
    P_i = \{ \bx = (x_1, x_2, \ldots, x_d) \colon i-1 \leqslant x_1 + x_2 + \ldots + x_d \leqslant i \}.
\end{align}
We show that copies of these polyhedra satisfy our requirements. First, by construction, the resulting polyhedra are convex. Second, the polyhedra $P_k$ and $P_{d+1-k}$ are symmetric with respect to the point $(\frac{1}{2}, \frac{1}{2}, \ldots, \frac{1}{2})$, and therefore can be translated to satisfy the condition $P_k = -P_{d+1-k}$.

We now demonstrate that any simplex of the form $n\Delta$ can be partitioned into copies of these polyhedra. We begin by partitioning the entire space $\mathbb{R}^d$ into unit cubes, and then further into copies of our polyhedra using hyperplanes of the form $\sum_{i=1}^d x_i = k$ for $k \in \mathbb{Z}$. Note that the hyperplanes containing the facets of $n\Delta$ participate in this dissection, and consequently the simplex itself is divided into polyhedra of the required type.

It remains to determine the number of copies of each polyhedron $P_k$. Observe that all copies of $P_k$ have the form $\bz + P_k$ for some $\bz = (z_1, z_2, \ldots, z_d) \in \mathbb{Z}^d$. For the copy $\bz + P_k$ to lie within the simplex $n\Delta$, it is necessary and sufficient that all $z_i \geqslant 0$ and $z_1 + z_2 + \ldots + z_d + k \leqslant n$. 

Thus, we need to count the number of non-negative integer solutions to the inequality $z_1 + z_2 + \ldots + z_d \leqslant n - k$. Introducing an auxiliary variable $z_0$, we equivalently count the number of non-negative integer solutions to the equation $z_0 + z_1 + z_2 + \ldots + z_d = n - k$. This count equals $\binom{n-k+d}{d}$, completing the proof.
\end{proof}

\begin{rem}\label{rem_Integer_translation}
    Note that the partition of the simplex $n\Delta$ constructed in the proof consists solely of integer polyhedra, with their copies differing from one another by integer vector translations. Consequently, for any integer simplex, the corresponding polyhedra will be integer polyhedra, and their copies will likewise differ by integer vectors. Moreover, the union of all constructed polyhedra yields the unit cube, which implies that for an integer simplex, the union will form an integer parallelepiped.
\end{rem}

We are now ready to prove Theorem~\ref{tm_Scaling_simplex}.

\begin{proof}[Proof of Theorem~\ref{tm_Scaling_simplex}]
Consider a simplex $T$. By Lemma~\ref{lm_Scaling_symplex}, there exist polyhedra $P_1, P_2, \ldots, P_d$ whose copies can tile any scaled simplex $nT$. Fix $n \in \mathbb{Z}_+$ and consider the partition of the simplex $nT$. Let $P_k^{(i)}$ denote the $i$-th copy of polyhedron $P_k$ for $i = 1, 2, \ldots, \binom{n-k+d}{d}$. Then
\begin{align}\label{eq_Saling_one}
    X_{nT} = \sum_{k=1}^d\sum_{i = 1}^{\binom{n-k+d}{d}} X_{P_k^{(i)}}.
\end{align}
    
According to Remark~\ref{rem_Integer_translation}, all $P_k^{(i)}$ are integer polyhedra, and any two copies $P_k^{(i)}$ and $P_k^{(j)}$ differ by an integer vector translation, hence $X_{P_k^{(i)}} = X_{P_k^{(j)}}$. Therefore, equality~\ref{eq_Saling_one} can be rewritten in the required form:
\begin{align}
    X_{nT} = \sum_{k=1}^d \binom{n-k+d}{d} X_{P_k}.
\end{align}

To complete the proof, it remains to verify that $X_{P_1} + X_{P_2} + \ldots + X_{P_d} = \text{const}$. By Remark~\ref{rem_Integer_translation}, the collection of polyhedra $P_1, P_2, \ldots, P_d$ forms a partition of an integer parallelepiped. Finally, observe that an integer parallelepiped always contains the same number of integer points, and
\begin{align}
    X_{P_1} + X_{P_2} + \ldots + X_{P_d} = X_{\cup P_i} = \text{const}.
\end{align}
\end{proof}


Theorem~\ref{tm_Scaling_polytope} follows directly from Theorem~\ref{tm_Scaling_simplex} and the fact that any convex polyhedron can be partitioned into simplices.

We now proceed to derive corollaries from these theorems. First, consider the three-dimensional case. In this setting, the statement of Theorem~\ref{tm_Scaling_polytope} takes the following form:
\begin{align}\label{eq_Cor_one}
    X_{nP} = \binom{n+2}{3}X_{P_1} + \binom{n+1}{3} X_{P_2} + \binom{n}{3}X_{P_3} \nonumber\\= \binom{n+2}{3}X_P + \binom{n+1}{3} X_{P_2} + \binom{n}{3}X_{-P}.
\end{align}
Since $X_P + X_{P_2} + X_{-P} = \text{const}$, we can subtract this constant from both sides of equation (\ref{eq_Cor_one}), thereby eliminating the second term on the right-hand side and obtaining Corollary~\ref{cor_3dim}:
\begin{align}\label{eq_Cor_two}
    X_{nP} \ceq \left(\binom{n+2}{3} - \binom{n+1}{3}\right) X_P + \left( \binom{n}{3} - \binom{n+1}{3}\right)X_{-P}\nonumber\\
    =\binom{n+1}{2}X_P - \binom{n}{2}X_{-P} = \frac{n(n+1)}{2}X_P - \frac{n(n-1)}{2}X_{-P}.
\end{align}

Note that if the polyhedron $P$ is centrally symmetric, then $X_P = X_{-P}$, and equality (\ref{eq_Cor_two}) can be simplified as follows:
\begin{align}\label{eq_Cor_three}
    X_{nP} \ceq \frac{n(n+1)}{2}X_P - \frac{n(n-1)}{2}X_{-P} = \left(\frac{n(n+1)}{2} - \frac{n(n-1)}{2} \right) X_P = nX_P.
\end{align}

Thus, Corollary~\ref{cor_3dim_sym} is proved.

We now turn to the four-dimensional case for centrally symmetric polyhedra. According to Theorem~\ref{tm_Scaling_polytope} and the symmetry condition, we have:
\begin{align}
    X_{nP} = \binom{n+3}{4}X_{P_1} + \binom{n+2}{4}X_{P_2} + \binom{n+1}{4} X_{P_3} + \binom{n}{4} X_{P_4}\nonumber\\
    = \binom{n+3}{4}X_{P} + \binom{n+2}{4}X_{P_2} + \binom{n+1}{4} X_{-P_2} + \binom{n}{4} X_{-P}\nonumber\\
    =\binom{n+3}{4}X_{P} + \binom{n+2}{4}X_{P_2} + \binom{n+1}{4} X_{P_2} + \binom{n}{4} X_{P}\nonumber\\
    =\left( \binom{n+3}{4} + \binom{n}{4} \right)X_P + \left( \binom{n+2}{4} + \binom{n+1}{4} \right) X_{P_2}\nonumber\\
    \ceq \left( \binom{n+3}{4} + \binom{n}{4} - \binom{n+2}{4} - \binom{n+1}{4} \right)X_P \nonumber\\= n^2X_P,
\end{align}


which proves Corollary~\ref{cor_4dim_sym}.

\section{Counterexamples}\label{sec_Counterexamples}
In this section we present counterexamples to statements that hold in dimension 2 but fail in higher dimensions.

We begin with an example of a centrally symmetric polyhedron in $\mathbb{R}^d$ that does not yield a constant random variable. Consider the simplex formed by the basis vectors:
\begin{align}\label{eq_Delta}
    \Delta = \{\bx = (x_1, x_2, \ldots, x_d) \in \mathbb{R}^d \colon x_1 + x_2 + \ldots + x_d \leqslant 1, x_i \geqslant 0\}.
\end{align}
Let $\Delta'$ denote the simplex symmetric to $\Delta$ with respect to the center of the cube $[0,1]^d$. Then the polyhedron $P = [0,1]^d \setminus (\Delta \cup \Delta')$ is the desired one. Indeed:
\begin{itemize}
    \item It is clearly convex and centrally symmetric  by construction
    \item Since $P \subset [0,1]^d$, it contains at most one integer point
    \item The union $\Delta \cup P \cup \Delta' = [0,1]^d$ always contains exactly one integer point
    \item $\Delta$ contains this point with non-zero probability
    \item Therefore, $P$ contains either 0 or 1 integer points with non-zero probabilities
\end{itemize}

\vspace{0.5cm}
We construct a counterexample to Property 5 of Statement~\ref{baza}. Consider the following construction: take a polyhedron $P \subset \mathbb{R}^d$ and embed it into $\mathbb{R}^{d+1}$ by adding a zero last coordinate. Then consider the new polyhedron $P \oplus e$, where $e = (0,0,\ldots, 0, 1)$. 

The polyhedron $P \oplus e$ contains exactly as many integer points from $\mathbb{Z}^{d+1}$ as its projection $P$ contains integer points from $\mathbb{Z}^d$. Therefore, if $P \subset \mathbb{R}^d$ satisfies $X_P \neq \text{const}$, then the new polyhedron $P \oplus e$ will also satisfy $X_{P \oplus e} \neq \text{const}$.

Taking $P$ to be the polyhedron from the previous example, and considering both $P$ and $e$ as polyhedra in $\mathbb{R}^{d+1}$, we observe that:
\begin{itemize}
    \item $X_P = X_e = 0$ almost surely
    \item Yet $X_{P \oplus e} \neq 0$
\end{itemize}
Thus Property 5 fails to hold.

\vspace{0.5cm}
As a counterexample to Property 7 of Statement~\ref{baza}, consider the simplex $\Delta$ defined in (\ref{eq_Delta}). This simplex contains:
\begin{itemize}
    \item 1 integer point with probability $\vol(\Delta) = \frac{1}{d!}$
    \item 0 integer points with probability $1 - \frac{1}{d!} = \frac{d!-1}{d!}$
\end{itemize}
Clearly, this distribution is not symmetric about its mean.

\section{Zonotopes and Constant Random Variables}\label{sec_Zonotope}
In this section we prove that zonotopes generate constant random variables. First, observe that integer parallelepipeds almost surely contain a fixed number of integer points. Indeed, let a parallelepiped be generated by vectors $v_1, v_2, \ldots, v_d$. Consider the sublattice $V \subset \mathbb{Z}^d$ spanned by these vectors. We define an equivalence relation on integer points $\mathbb{Z}^d$ via the rule:
\begin{align}
    z_1 \sim z_2 \iff z_1 - z_2 \in V.
\end{align}

Then, if we translate the original parallelepiped so that it contains no integer points on its facets, it will contain exactly one representative from each equivalence class. It remains to observe that the probability of having integer points on the facets of the parallelepiped is zero.

Now let us consider integer zonotopes. Take a zonotope $Z = v_1 \oplus v_2 \oplus \ldots \oplus v_n$. This zonotope can be partitioned into $\binom{n}{d}$ parallelepipeds generated by all subsets of $d$ vectors (see \cite{McMullen}). Since the vectors $v_1, v_2, \ldots, v_n$ are integer, all parallelepipeds in the partition are integer and contain a fixed number of integer points. Consequently, the zonotope itself must contain a fixed number of integer points.


\section{Reeve Tetrahedron}\label{sec_Reeve}
In this section we compute the variance for Reeve tetrahedra, thereby demonstrating that in dimension 3, the variance is not a combinatorial invariant of polyhedra.

Fix $n \in \mathbb{Z}_+$ and consider the tetrahedron $T_n$ with vertices at $(0,0,0)$, $(0,1,0)$, $(1,0,0)$, and $(1,1,n)$. We will count the interior points of $T_n$ separately for each \textit{layer} $E_k = \{(x_1, x_2, \ldots, x_d) \in \mathbb{R}^d \colon x_d = k\}$. To do this, we introduce for each layer an indicator that the tetrahedron contains an integer point on that layer:
\begin{align}
    I_k = |(T_n + X) \cap E_k \cap \Z^d|.
\end{align}
Then $X_{T_n} = \sum_{k = 1}^{n} I_k$. We are interested in the quantity
\begin{align}
    \Var X_{T_n} = \sum_{i = 1}^n \sum_{j = 1}^n \cov(I_i, I_j).
\end{align} 

For convenience, let us introduce notation for the tetrahedron's section at each level: $S_k = (T_n + X) \cap E_k$. Let $X = (U, V, W)$, where $U, V, W$ are independent random variables uniformly distributed on $[0,1]$. Then $S_k$ is a triangle with vertices:
\begin{align*}
\left(\frac{k - W}{n} + U, 1 + V, k\right), \quad 
\left(1 + U, \frac{k - W}{n} + V, k\right), \quad 
\left(\frac{k - W}{n} + U, \frac{k - W}{n} + V, k\right).
\end{align*}

Note that for fixed $W = w$, $S_k$ becomes a triangle with vertices:
\begin{align*}
\left(\frac{k - w}{n}, 1, k\right), \quad 
\left(1, \frac{k - w}{n}, k\right), \quad 
\left(\frac{k - w}{n}, \frac{k - w}{n}, k\right),
\end{align*}
shifted by vector $(U, V, 0)$. This triangle lies entirely within the unit square and, when shifted, can only contain the point $(1,1,k)$. Therefore, $I_k$ is indeed an indicator variable taking values 0 or 1. In particular, $\mathbb{E} I_k^2 = \mathbb{E} I_k$.

For further calculations, we need to compute the sum $\sum_{k=1}^n \mathbb{E} I_k$. This can be done without explicitly computing each expectation: since $X_{T_n} = \sum_{k=1}^n I_k$, we have 
\begin{align*}
\sum_{k=1}^n \mathbb{E} I_k = \mathbb{E} X_{T_n} = \vol (T_n) = \frac{n}{6}.
\end{align*}

Next we compute $\mathbb{E}[I_k I_l]$. Fix $k < l$ and $W = w$. Consider two triangles in the coordinate plane $x_3 = 0$:
\begin{align}
    R_k& = \conv \left\{ \left(\frac{k-w}{n}, 1, 0\right);~\left(1, \frac{k-w}{n}, 0\right);~ \left(\frac{k-w}{n}, \frac{k-w}{n}, 0\right) \right\},\nonumber\\
    R_l& = \conv \left\{ \left(\frac{l-w}{n}, 1, 0\right);~\left(1, \frac{l-w}{n}, 0\right);~ \left(\frac{l-w}{n}, \frac{l-w}{n}, 0\right) \right\}.
\end{align}
Then $S_k = R_k + (U, V, k)$ and $S_l = R_l + (U, V, l)$. For fixed $W = w$, we have:
\begin{align}
    I_kI_l = \ind_{\{I_k = 1,~ I_l = 1\}} = \ind_{\{ (1, 1, k) \in S_k,~ (1, 1, l) \in S_l \}} = \ind_{\{ (1, 1, 0) \in R_k + (U, V, 0),~ (1, 1, 0) \in R_l + (U, V, 0) \}}\nonumber\\
    = \ind_{\{ (1, 1, 0) \in R_k \cap R_l + (U, V, 0) \}}.
\end{align}
Thus,
\begin{align}
    \mathbb{E}[I_kI_l|W = w] &= \mathbb{E} \mathbbm{1}_{\{ (1, 1, 0) \in R_k \cap R_l + (U, V, 0) \}} = \mathbb{P}\{ (1, 1, 0) \in R_k \cap R_l + (U, V, 0) \} \nonumber\\
    &= \text{area}(R_k \cap R_l).
\end{align}

Since $k < l$, the intersection $R_k \cap R_l \neq \varnothing$ if and only if the right angle vertex of $R_l$ lies below the hypotenuse of $R_k$, i.e., when $2(\frac{l-w}{n}) \leqslant \frac{k-w}{n} + 1$ or equivalently $2l \leqslant k + w + n$. In this case, the intersection area equals:
\begin{align}
    \text{area}(R_k \cap R_l) = \frac{(k + w + n - 2l)^2}{2n^2}.
\end{align}
Therefore, the conditional expectation becomes:
\begin{align}
    \mathbb{E}[I_kI_l|W = w] = \frac{(k + w + n - 2l)^2}{2n^2}\cdot \mathbbm{1}_{\{ 2l \leqslant k + w + n \}}.
\end{align}

Consequently, the unconditional expectation is:
\begin{align}
    \mathbb{E} I_kI_l = \int_0^1 \mathbb{E}[I_kI_l|W = w] \,dw = \int_0^1 \frac{(k + w + n - 2l)^2}{2n^2}\cdot \mathbbm{1}_{\{ 2l \leqslant k + w + n \}}\,dw.
\end{align}

Note that when $l > \frac{k + n}{2}$, we have $\mathbb{E} I_kI_l = 0$. Otherwise, the expectation evaluates to:
\begin{align}
    \mathbb{E} I_kI_l = \int_0^1 \frac{(k + w + n - 2l)^2}{2n^2}\,dw = \frac{-(k - 2l + n)^3 + (1 + k - 2l + n)^3}{6n^2}.
\end{align}

We now have all components needed to compute the variance of the Reeve tetrahedron. The calculation proceeds as follows:
\begin{multline*}
    \Var X_{T_n} = \sum_{k=1}^n \sum_{l=1}^n \cov(I_k, I_l) 
    = \sum_{k \neq l} (\mathbb{E} I_kI_l - \mathbb{E} I_k \mathbb{E} I_l) + \sum_{k=1}^n (\mathbb{E} I_k^2 - (\mathbb{E} I_k)^2) \\
    = 2 \sum_{k < l} \mathbb{E} I_kI_l - \sum_{k \neq l}\mathbb{E} I_k \mathbb{E} I_l + \sum_{k=1}^n (\mathbb{E} I_k - (\mathbb{E} I_k)^2) \\
    = 2 \sum_{k < l \leqslant \frac{n+k}{2}} \mathbb{E} I_kI_l + \sum_{k=1}^n \mathbb{E} I_k - \left(\sum_{k=1}^n \mathbb{E} I_k\right)\left(\sum_{l=1}^n \mathbb{E} I_l\right) \\
    = 2 \sum_{k < l \leqslant \frac{n+k}{2}} \frac{-(k - 2l + n)^3 + (1 + k - 2l + n)^3}{6n^2} + \frac{n}{6} - \frac{n^2}{36} \\
    = \frac{n^3 - 4n^2 + 4n - 1}{24n} + \frac{n}{6} - \frac{n^2}{36} 
    = \frac{n^3 + 12n - 3}{72n}.
\end{multline*}

\end{document}